\documentclass[11pt]{article}
\usepackage{amssymb,amsthm,amsmath,bm}
\usepackage{indentfirst}
\usepackage[numbers,sort&compress]{natbib}
\usepackage{enumerate}
\usepackage{amscd}
\usepackage{array}
\usepackage{amsmath}
\usepackage{latexsym}
\usepackage{amsfonts}
\usepackage{amssymb}
\usepackage{amsthm}
\usepackage{verbatim}
\usepackage{mathrsfs}
\usepackage{enumerate}
\usepackage{hyperref}

\theoremstyle{plain}\newtheorem{definition}{Definition}[section]
\theoremstyle{definition}\newtheorem{theorem}{Theorem}[section]
\theoremstyle{plain}\newtheorem{lemma}[theorem]{Lemma}
\theoremstyle{plain}
\theoremstyle{plain}
\theoremstyle{remark}\newtheorem{remark}{Remark}[section]
\theoremstyle{definition}
\theoremstyle{plain}
\usepackage{xcolor}

\def\no{\noindent}

\newcommand{\norm}[1]{\left\|#1\right\|}
\newcommand{\Div}{\mathrm{div}\,}
\newcommand{\B}{\Big}

\newcommand{\be}{\begin{equation}}
\newcommand{\ee}{\end{equation}}
 \newcommand{\ba}{\begin{aligned}}
 \newcommand{\ea}{\end{aligned}}

  \newcommand{\f}{\frac}
    
  \newcommand{\ben}{\begin{enumerate}}
   \newcommand{\een}{\end{enumerate}}

\newcommand{\Rmnum}[1]{\expandafter\@slowromancap\romannumeral #1@}

\allowdisplaybreaks

\numberwithin{equation}{section}
\begin{document}
\title{Energy conservation of weak solutions for the incompressible Euler equations via vorticity}
 \author{Jitao Liu,\footnote{Department of Mathematics, Faculty of Science, Beijing University of Technology, Beijing, 100124, China  Email: jtliu@bjut.edu.cn, jtliumath@qq.com} ~~
 Yanqing Wang,\footnote{   College of Mathematics and   Information Science, Zhengzhou University of Light Industry, Zhengzhou, Henan  450002,  P. R. China Email: wangyanqing20056@gmail.com}
 ~~~and~~~Yulin Ye\footnote{School of Mathematics and Statistics,
Henan University,
Kaifeng, 475004,
P. R. China. Email: ylye@vip.henu.edu.cn}}
\date{}
\maketitle
\begin{abstract}
Motivated by the works of  Cheskidov, Lopes Filho, Nussenzveig Lopes and Shvydkoy in \cite[{\it Commun. Math. Phys.} 348: 129-143, 2016]{[CLNS]} and Chen and Yu in
 \cite[{\it J. Math. Pures Appl.} 131: 1-16, 2019]{[CY]},
 we address how the $L^p$ control of vorticity could influence the energy conservation for the incompressible homogeneous and nonhomogeneous Euler equations in this paper. For the homogeneous flow in the periodic domain or whole space, we provide a self-contained proof for the criterion $\omega=\text{curl}u\in L^{3}(0,T;L^{\f{3n}{n+2}}(\Omega))\,(n=2,3)$, which generalizes the corresponding result in  \cite{[CLNS]} and can be viewed as  in Onsager critical   spatio-temporal spaces. Regarding the nonhomogeneous flow,
 it is shown  that the energy is conserved as long as the vorticity lies in the same space as before and $\nabla\sqrt{\rho}$ belongs to $L^{\infty}(0,T;L^{n}(\mathbb{T}^{n}))\,(n=2,3)$, which
   gives an affirmative answer to a problem proposed by  Chen and Yu in
 \cite{[CY]}.
\end{abstract}
\section{Introduction}
\label{intro}

The incompressible homogeneous Euler  equations describing the motion of  inviscid  fluid with constant density are governed by
\begin{equation}\left\{\begin{aligned}\label{Euler}
&v_{t}+v\cdot \nabla v+\nabla \pi=0,~~\text{in}~~~~\Omega\times(0,\,T),\\ &\text{div}\, v=0,~~~~~~~~~~~~~~~~~~\text{in}~~~~\Omega\times(0,\,T),\\
&v|_{t=0}=v_{0}(x)~~~~~~~~~~~~~~\text{on} ~~~ \Omega\times\{t=0\},
\end{aligned}\right.\end{equation}
where $v$ and $\pi$ stands for the velocity field and
 the  pressure of the fluid, respectively. Let $\Omega$ denote the whole space $\mathbb{R}^{n}$ or periodic domain $\mathbb{T}^{n}$ with the spatial dimension $n=2$ or $n=3$. Moreover, we denote the vorticity of velocity field $v$ by $\omega=\text{curl}\,v$.

It is well-known that, for any smooth velocity field $v$, the $L^{2}$ energy of incompressible homogeneous Euler equations \eqref{Euler} is conserved as follows
\be\label{ec}
\int_{\Omega}|v(x,t)|^{2}dx=\int_{\Omega}|v(x,0)|^{2}dx, ~~\text{for all}~~ t>0.
\ee
However, if $v$ is not smooth enough, the energy equality may not hold. Therefore, there are many substantial effort contributing to investigate the relation between energy
conservation and regularity of weak solutions over the past years. This issue is closely connected to Onsager's seminal work \cite{[Onsager]}, where he conjectured that the weak solutions with H\"older continuity exponent $\alpha>\frac{1}{3}$ do conserve energy and turbulent or anomalous dissipation occurs when H\"older continuity exponent $\alpha\leq \frac{1}{3}$.
We call the two different parts of Onsager's conjecture as {\it positive part} and {\it negative part}.

For the {\it positive part}, the first rigorous proof was presented by Eyink in \cite{[GLE]}, where it is proved that conservation holds in a subset of $C^{\alpha}$ with $\alpha >\frac{1}{3}$. Subsequently,
Constantin-E-Titi \cite{[CET]} expanded this class to  Besov space $L^{3}(0,T; B^{\alpha}_{3,\infty}(\mathbb{T}^{3}))$ with $\alpha>1/3$. Later, Cheskidov-Constantin-Friedlander-Shvydkoy \cite{[CCFS]} generalized it to a wider range of space, i.e. the critical space $L^{3}(0,T; B^{1/3}_{3,c(\mathbb{N})})$, where  $B^{1/3}_{3,c(\mathbb{N})}=\{v\in B^{1/3}_{3,\infty}, \lim_{q\rightarrow\infty}2^{q}\|\Delta_{q}v\|^{3}_{L^{3}}=0\}$ and $\Delta_{q}$ represents a smooth restriction of $v$ into Fourier modes of order $2^q$. Along this direction, there are some progresses recently, one can refer to \cite{[FW2018],[BGSTW]} for details. Regarding the {\it negative part}, through the application of convex integration argument, it has been widely investigated by many mathematicians, such as De Lellis and  Szekelyhidi \cite{[DS0],[DS1],[DS2],[DS3]}, Isett \cite{[Isett]} and so on.

As everyone knows, in fluid mechanics, the vorticity is an important physical quantity(see \cite{[MB]}). However, we can notice that all above control is about velocity field, there is no control on the integrability of vorticity. In \cite{[CLNS]}, for the two dimensional flow, Cheskidov, Lopes Filho, Nussenzveig Lopes and Shvydkoy found that
the energy is conserved if the vorticity  $\omega$ lies in $ L^{\infty}(0,T;L^{\f32}(\mathbb{T}^2))$.
In addition, they also constructed an example of velocity field in the class $\omega\in L^{(\f32)^{-}}(\mathbb{T}^2)$ and $ v\in B^{\f13}_{3,\infty}(\mathbb{T}^2)$
with non-vanishing energy flux, which shows that the Onsager regularity threshold does not relax even with an additional $L^p$-control on vorticity. Nevertheless, compared with above control on velocity field, the time control on the vorticity is only $L^\infty(0,T)$, other than $L^3(0,T)$.  A natural question is whether
the time direction can be relaxed.  Recall the classical embedding relations
\be\label{inclusion relationship}
\dot{W}^{1,\f32}(\mathbb{T}^2)\hookrightarrow \dot{H}^{\f23}(\mathbb{T}^2)\approx \dot{B}^{\f23}_{2,2}(\mathbb{T}^2)
\hookrightarrow\dot{B}^{\f13}_{3,c(\mathbb{N})}(\mathbb{T}^2),
\ee
and the aforementioned Onsager's result  $L^{3}(0,T; B^{1/3}_{3,c(\mathbb{N})})$ due to Cheskidov-Constantin-Friedlander-Shvydkoy, we notice that the control on vorticity should be relaxed to $\omega \in L^{3}(0,T;L^{\f32}(\mathbb{T}^2))$ theoretically. In current paper, for both two and three dimensional flow, periodic domain and whole space, we will present a self-contained proof for the criterion $\omega\in L^{3}(0,T;L^{\f{3n}{n+2}}(\Omega))$. We state our first result as follows.
\begin{theorem}\label{the1.1} Let  $v$ be a  weak solution of incompressible homogeneous Euler equations \eqref{Euler} in the sense of Definition \ref{eulerdefi}. Assume the vorticity
\be\label{lwy}
\omega\in L^{3}(0,T;L^{\f{3n}{n+2}}(\Omega)),\ee
 for $\Omega=\mathbb{T}^{n}$ or $\Omega=\mathbb{R}^{n}$ with $n=2,3$,
then the energy of $v$ is preserved, that is, for any $t\in [0,T]$,
$$\|v(t)\|_{L^{2}(\Omega)}=\|v(0)\|_{L^{2}(\Omega)}.$$
\end{theorem}
\begin{remark}
This theorem is an improvement of Cheskidov, Lopes Filho, Nussenzveig Lopes and Shvydkoy's result in \cite{[CLNS]} when the dimension $n=2$.
\end{remark}
\begin{remark}
To our best knowledge, for three dimensional flow, the sufficient condition $\omega\in L^{3}(0,T;L^{\f{9}{5}}(\Omega))$ for energy conservation should be the first control on the integrability of vorticity in the $L^p$ frame work.
\end{remark}
\begin{remark}
In fact, in three dimensional flow,  the following inclusion relationship
 \begin{equation}\label{3d}
 \dot{W}^{1,\frac{9}{5}} \hookrightarrow \dot{H}^{\f56} \hookrightarrow \dot{B}^{\f56}_{2,2} \hookrightarrow \dot{B}^{\frac{1}{3}}_{3,c(\mathbb{N})}
 \end{equation}
is valid. In consideration of Cheskidov-Constantin-Friedlander-Shvydkoy's work, \eqref{3d} together with \eqref{inclusion relationship} indicates that the integrability exponent of vorticity is in {\it Onsager critical}  spaces.
\end{remark}
\begin{remark}
By means of the generalized Biot-Savart law for the  incompressible flows
$$
\partial_{i}\partial_{i}u_{j}=\partial_{i}(\partial_{j}u_{i}+\partial_{i}u_{j}),
$$
the vorticity
$\omega $  can be replaced  by  deformation tensor  $\mathcal{D}(u)=\f{1}{2}(\nabla u+\nabla u^{^{\text{T}}})$ here.
\end{remark}

For the three dimensional incompressible homogeneous Navier-Stokes equations on the bounded domain, the energy equality  criterion  $\nabla u\in L^{3} (0,T;L^{\f95}(\Omega))$ is recently obtained by Berselli and  Chiodaroli  in \cite{[BC]} (see \cite{[WMH]}  for the Cauchy problem and   related work \cite{[Shvydkoy]}). It is worth remarking that the condition \eqref{lwy} leads to
$$
v\in L^{3}(0,T; L^{\f{3n}{n-1}}(\Omega)),$$
which are the Onsager-critical spaces defined by  Shvydkoy  for the incompressible homogeneous Euler equations  in \cite{[Shvydkoy]}. In addition,
the proof of Theorem \ref{the1.1} is partial motived by  Cheskidov, Lopes Filho,  Nussenzveig Lopes and  Shvydkoy's work \cite{[CLNS]} and recent energy balance criterion based on a combination of velocity field and its gradient for the weak solutions of Navier-Stokes equations in \cite{[WY]}.  For the whole spaces $\mathbb{R}^{n}$ and periodic domain $\mathbb{T}^{n}$, we will provide two  completely different proof.

In the second part of this paper, we will consider the following incompressible nonhomogeneous Euler equations that describes the inviscid fluids with non-constant density:
\be\left\{\ba\label{NEuler}
&\rho_t+\nabla \cdot (\rho v)=0, \\
&(\rho v)_{t} +\Div(\rho v\otimes v)+\nabla
\pi=0, \\
&\Div v=0.
 \ea\right.\ee
In real world applications, fluctuations of density are widely present in turbulent flows, which is because different species concentrations, temperature variation or pressure.
For example, in geophysical fluids, when external forces are present (e.g. gravity), density stratification may occurs due to temperature and
salinity gradients (ocean) or moisture effects (atmosphere). This phenomenon needs to be described by the model \eqref{NEuler}, other than \eqref{Euler}.

Compared to the homogeneous flows, the presence of density will bring additional challenges. In \cite{[LS]}, Leslie-Shvydkoy  studied the energy conservation of weak solutions for the incompressible nonhomogeneous Euler equations \eqref{NEuler} and showed that if the weak solutions
satisfy
$$\ba
&0<c_{1}\leq \rho\leq c_{2}<\infty,\rho\in L^{a}(0,T;B^{\f13}_{a,\infty}),\\
&v\in L^{b}(0,T;B^{\f13}_{b,c(\mathbb{N})}),\pi\in L^{\f{b}{2}}(0,T;B^{\f13}_{\f{b}{2},\infty}),\f1a+\f3b=1,b\geq3,
\ea$$
then $v$ conserves energy. To handle the vacuum, Feireisl-Gwiazda-Gwiazda-Wiedemann \cite{[EGSW]}  used Constantin-E-Titi type
commutators on mollifying kernel to derive the following energy conservation criterion:
$$
v \in B_{p, \infty}^{\alpha} (0, T;\mathbb{T}^{n} ), \quad \rho, \rho v \in B_{\f{p}{p-2}, \infty}^{\beta} (0, T;\mathbb{T}^{n} ), \quad \pi \in L_{l o c}^{\f{p}{p-1}} (0, T; \mathbb{T}^{n} ),
$$
with  $2 \alpha+\beta>1$ and $0 \leq \alpha, \beta \leq 1$. By means of the Lions type commutators on mollifying kernel, Chen and Yu \cite{[CY]} established the sufficient conditions allowing vacuum:
\be\ba\label{cy}
&\rho \in L^{\infty} (0, T;\mathbb{T}^{n} ) \cap L^{\f{q}{q-3}} (0, T; W^{1, \f{q}{q-3}} (\mathbb{T}^{n} ) ), \rho_{t} \in L^{\f{q}{q-3}} (0, T;\mathbb{T}^{n} ),\\
&v\in L^{q} (0, T ; B_{q, \infty}^{\alpha} (\mathbb{T}^{n} ) ), \alpha>1/3.
 \ea\ee
Later, in \cite{[NNT2]}, Nguyen, Nguyen and Tang obtained the energy conservation criterion as follows:
\be\ba\notag
&0\leq \rho \in L^{\infty} (0, T;\mathbb{T}^{n} ) \cap L^{\infty} (0, T ; {\mathcal{V}}_{\delta_0}^{\f23, \infty} (\mathbb{T}^{n} ) ), \rho^{-1} \in L^{p} (0, T;\mathbb{T}^{n} ),{\bf 1}_{\rho\leq\delta_0}\partial_t\rho\in L^{q} (0, T;\mathbb{T}^{n} ),\\
&v\in L^{3} (0, T;\mathbb{T}^{n} ) \cap L^{3} (0, T ; {\mathcal{V}}_{\delta_0}^{\f13, 3} (\mathbb{T}^{n} ) ), \limsup_{\delta\rightarrow0}\|v\|_{ L^{3} (0, T ; {\mathcal{V}}_{\delta_0}^{\f13, 3} (\mathbb{T}^{n} ) )}=0,\pi\in L^{\f32} (0, T;\mathbb{T}^{n} ).
 \ea\ee
 where $\delta_0>0$ and $p>0,\,q\geq3$ such that $\f1p+\f1q=\f13$, $\|f\|_{{\mathcal{V}}_{\delta}^{\beta, p} (\mathbb{T}^{n})}=\sup_{|h|<\delta}|h|^{-\beta}\|f(\cdot+h)-f(\cdot)\|_{L^p(\mathbb{T}^{n} )}$ and $\|f\|_{ L^{q} (0, T ; {\mathcal{V}}_{\delta}^{\beta, p} (\mathbb{T}^{n}))}=\Big(\int_0^T\|f(t)\|_{{\mathcal{V}}_{\delta}^{\beta, p} (\mathbb{T}^{n})}^q\Big)^{\f1q}$
 for $\delta>0,\,\beta>0,\,p\geq1$ and $1\leq q\leq\infty$.

In
 \cite[Remark 1.1 (4), p.5]{[CY]}, Chen and Yu   proposed a question:
 {\it  It would be interesting to try to understand how an $L^p$ control of the vorticity could affect the energy conservation for density-dependent incompressible Euler system.  }
Indeed, by taking curl on both sides of \eqref{NEuler}, one can derive the vorticity equations of two dimensional flows as
\begin{equation}\label{2dv}
\partial_tw +u\cdot \nabla  w=\f{\nabla^{\bot}\rho\cdot\nabla\pi}{{\rho}^2},
\end{equation}
and three dimensional flow as
\begin{equation}\label{3dv}
\partial_tw +u\cdot \nabla w=w\cdot \nabla u+\f{\nabla\rho\times\nabla\pi}{{\rho}^2},
\end{equation}
respectively, where $\nabla^{\bot}=(\partial_2,-\partial_1)$. Due to the presence of nonlinear term involving density and pressure, we can discover that, for two dimensional flows, the voriticity is not conserved from initial data any more, which is quite different from homogenous flows. Moreover, the vorticity may be produced from density and pressure gradients (baroclinic torque). All these phenomena reveal the fluctuation of energy may come from the loss of regularity of the density.

In this paper,  we will give a complete answer to  Chen-Yu's issue proposed in \cite{[CY]} for both two and three dimensional flow.  When we were preparing this paper, we notice that, for the two dimensional model \eqref{NEuler}, Chen \cite{[Chen]} derived the sufficient conditions as follows
\be\ba\label{chen}
&\rho \in L^{\infty}([0,T]\times\mathbb{T}^{2})\cap L^{2}(0,T;H^{1}(\mathbb{T}^{2})),\omega \in L^{6}(0,T;L^{\f32}(\mathbb{T}^{2})),\\
&\sqrt{\rho}v\in L^{\infty}(0,T;L^{2}(\mathbb{T}^{2})).
\ea
\ee
Similar with the homogenous flow, one can found that, the time integrity exponent on vorticity 6 in \eqref{chen} is not Onsager critical in the $L^p$ frame work. Besides, for the more widely used three dimensional flows in real world applications, it is not considered in \cite{[Chen]}. In fact, from the vorticity equations \eqref{2dv} and \eqref{3dv}, we can learn that the three dimensional flow is much more complicated than two dimensional one, due to the presence of stretching term $w\cdot \nabla u$. In present work, we will entirely solve the above unfinished issues. To clarify the effect of density on conservation of energy, we consider the cases without vacuum and allowing vacuum both. Now, we show the first result for the case without vacuum.

\begin{theorem}\label{the1.2}
Let $\Omega =\mathbb{T}^n$ with $n=2,3$. Then for any $t\in [0,T]$, the energy
 $$
 \frac{1}{2}\int_{\mathbb{T}^n}\rho(x,t)|v(x,t)|^{2}dx=\frac{1}{2}\int_{\mathbb{T}^n}\rho(x,0)|v(x,0)|^{2}dx,
  $$
 is conserved for any weak solution  $(\rho, v)$  to the incompressible nonhomogeneous Euler equations \eqref{NEuler} as long as
 \be\label{key01}
\ba
&0<c_{1}\leq \rho\leq c_{2}<\infty,  \omega\in L^{3}(0,T;L^{\f{3n}{n+2}}(\mathbb{T}^{n})), \pi\in L^{\f32}(0,T;L^{\f{3n}{2n-2}}(\mathbb{T}^n) ),\\
&{\text and}\ \sqrt{\rho_0}v_0\in L^{2+\delta} ~\text{for~ any}~ \delta>0,
\ea\ee
where $c_1,c_2$ are two positive constants.
\end{theorem}

Compared with Theorem \ref{the1.1}, Theorem \ref{the1.2} is much more involved due to the presence of density in fluid. Because the vacuum is absent, to avoid the commutators in time, we use $\frac{(\rho v)^\varepsilon}{\rho^\varepsilon}$ as the test function in the spirit of \cite{[LS],[NNT]}. Here the function $(\cdot)^\varepsilon$ is mollified only in spatial direction. With the help of Constantin-E-Titi type commutators on mollifying kernel in \cite{[WY],[NNT]} and the fact
$v\in L^{3}(0,T;W^{1,\f{3n}{n+2}}(\mathbb{T}^{n}))$, we can successfully prove Theorem \ref{the1.2}. More importantly, we will show that the result in Theorem \ref{the1.2} also holds for the flow allowing vacuum (see Theorem \ref{the1.3}). Rearding the vacuum case, the quantity $\frac{(\rho v)^\varepsilon}{\rho^\varepsilon}$ makes no sense in
the presence of regions of vacuum and hence can not be chosen as test function any more. To over this difficulty, we mollify the velocity field both in time and space. Besides this,  the other key point is the application of Lions type Poincar\'e inequality in \cite[Remark 5.1, p4]{[Lions2]} and
\cite[Lemma 3.2, page 47]{[Feireisl2004]}, which was noticed in \cite{[YWW]}. The second result for the flow allowing vacuum is stated as follows.
\begin{theorem}\label{the1.3}
Let $\Omega=\mathbb{T}^{n}$
 with $n=2,3$, then the energy equality for any weak solution $(\rho, v)$ to the incompressible nonhomogeneous Euler equation \eqref{NEuler} is valid provided
 \be\label{key02}
\ba
&0\leq \rho\leq c<\infty, ~\nabla\sqrt{\rho}\in L^{\infty}(0,T;L^{n}(\mathbb{T}^{n})),\\& \omega\in L^{3}(0,T;L^{\f{3n}{n+2}}(\mathbb{T}^{n})),\text {and}\ v_0\in L^{2+\delta}(\mathbb{T}^{n})\  \text {for any }\ \delta>0.
\ea\ee
\end{theorem}

\begin{remark}
It seems that this theorem is  the first attempt to establish the control on the integrability of vorticity for energy conservation in the 3D incompressible nonhomogeneous  Euler equations.
Theorem \ref{the1.3}
gives a positive answer to  Chen-Yu's issue arising  in \cite{[CY]}.
When the density is a constant,  Theorem  \ref{the1.2} and Theorem \ref{the1.3} reduce  to Theorem  \ref{the1.1}.
\end{remark}

\begin{remark}
With Lemma \ref{lem4.1} in hand, by a slight modified the proof of  \eqref{3.81} and \eqref{3.11}, one can immediately derive \eqref{chen} for the case allowing vacuum.
\end{remark}

The rest of this paper is divided into four sections.
 In Section 2, we present some notations and several useful auxiliary lemmas.
 Section 3 contains two brief different proof of Theorem \ref{the1.1} for the periodic domain and whole space.
 Section 4 is devoted to the energy conservation of weak solutions for incompressible nonhomogeneous Euler equations via vorticity.

\section{Notations and some auxiliary lemmas}

In this section, we introduce the notations and some useful auxiliary lemmas which will play a foundamental role in this paper.

{\bf Sobolev spaces:} For $p\in [1,\,\infty]$, we use the space $L^{p}(0,\,T;X)$ to stand for the set of measurable functions $f(x,t)$ on the interval $(0,\,T)$ with values in $X$ and $\|f(\cdot,t)\|_{X}$ belongs to $L^{p}(0,\,T)$. The classical Sobolev space  $W^{k,p}(\Omega)$
is equipped with the norm $\|f\|_{W^{k,p}(\Omega)}=\sum\limits_{|\alpha| =0}^{k}\|D^{\alpha}f\|_{L^{p}(\Omega)}$ and $\langle f,\,g\rangle$ denotes the standard $L^{2}$ inner product between $f$ and $g$. As usual, for any Schwartz function $f$ on $\mathbb{R}^{n}$, we define the Fourier transform of $f$ by $\hat{f}$ with $\hat{f}(\xi)=\frac{1}{(2 \pi)^{\frac{n}{2}}}\int_{\mathbb{R}^{n}}f (x)e^{- i\xi\cdot x}\,dx$. The inverse Fourier transform $f^{\vee}$ is defined as $f^{\vee}(\xi)=\widehat{f}(-\xi)$.

{\bf Besov spaces:}
 To define Besov  spaces, we need the following dyadic unity partition
(see e.g. \cite{[BCD]}). Choose two nonnegative radial
functions $\varrho$, $\varphi\in C^{\infty}(\mathbb{R}^{n})$
supported respectively in the ball $\{\xi\in
\mathbb{R}^{n}:|\xi|\leq \frac{3}{4} \}$ and the shell $\{\xi\in
\mathbb{R}^{n}: \frac{3}{4}\leq |\xi|\leq
  \frac{8}{3} \}$ such that
\begin{equation*}
 \varrho(\xi)+\sum_{j\geq 0}\varphi(2^{-j}\xi)=1, \quad
 \forall\xi\in\mathbb{R}^{n}; \qquad
 \sum_{j\in \mathbb{Z}}\varphi(2^{-j}\xi)=1, \quad \forall\xi\neq 0.
\end{equation*}
Then for every $\xi\in\mathbb{R}^{n},$ $\varphi(\xi)=\varrho(\xi/2)-\varrho(\xi)$. Denote $h=\mathcal{F}^{-1} \varphi $ and $\tilde{h}=\mathcal{F}^{-1}\varrho$, then nonhomogeneous dyadic blocks  $\Delta_{j}$ are defined by
$$
\Delta_{j} u:=0 ~~ \text{if} ~~ j \leq-2, ~~ \Delta_{-1} u:=\varrho(D) u =\int_{\mathbb{R}^n}\tilde{h}(y)u(x-y)dy,$$
$$\text{and}~~\Delta_{j} u:=\varphi\left(2^{-j} D\right) u=2^{jn}\int_{\mathbb{R}^n}h(2^{j}y)u(x-y)dy  ~~\text{if}~~ j \geq 0.
$$
The nonhomogeneous low-frequency cut-off operator $S_j$ is defined by
$$
S_{j}u:= \sum_{k\leq j-1}\Delta_{k}u.$$
The homogeneous dyadic blocks $\dot{\Delta}_{j}$ and homogeneous low-frequency cut-off operators $\dot{S}_j$ are  defined  for every $j\in\mathbb{Z}$ by
\begin{equation*}
  \dot{\Delta}_{j}u:= \varphi(2^{-j}D)u=2^{jn}\int_{\mathbb{R}^n}h(2^{j}y)u(x-y)dy,
\end{equation*}
$$ \text { and }~~ \dot{S}_{j}u:=\varrho(2^{-j}D)u=2^{jn}\int_{\mathbb{R}^n}\tilde{h}(2^{j}y)u(x-y)dy.$$
Then for $-\infty <s<\infty$ and $1\leq p,q\leq \infty,$ the homogeneous Besov semi-norm $ \|f\|_{\dot{B}^{s}_{p, q}}$ of $f\in \mathcal{S}'/\mathcal{P}$ is given by
\begin{equation*}
	\begin{aligned}
  \norm{f}_{\dot{B}^{s}_{p, q}}:=\left\{\begin{array}{lll}\left(\sum_{j\in \mathbb{Z}}2^{jqs}\norm{\dot{\Delta}_{j} f} _{L^p}^q\right)^{1/q},~~\text{if}\ q\in [1,\infty),\\
  	\sup_{j\in \mathbb{Z}}2^{js}\norm {\dot{\Delta}_{j}f} _{L^p},~~~~~~~~~\text{if}~q=\infty.\end{array}\right.
\end{aligned}\end{equation*}
Moreover, for $s\in\mathbb{R}$ and $1\leq p,q\leq \infty$, we define the nonhomogeneous Besov norm $\norm{f}_{B^s_{p,q}}$ of $f\in \mathcal{S}^{'}$ as
$$\norm{f}_{B^s_{p,q}}=\norm{f}_{{L^p}}+\norm{f}_{\dot{B}^s_{p,q}}.$$
Motivated by \cite{[CCFS]}, we denote $\dot{B}^\alpha _{p,c(\mathbb{N})}$ with $\alpha\in\mathbb{R}$ and $1\leq p \leq \infty$ as the class of all tempered distributions $f$ satisfying
\begin{equation}\label{2.1}
\norm{f}_{\dot{B}^\alpha _{p,\infty}}<\infty~ \text{and}~ 	\lim_{j\rightarrow \infty} 2^{j\alpha}\norm{\dot{\Delta}_j f}_{L^p}=0.
\end{equation}
 The Littlewood-Paley decomposition and Besov space for the periodic domain $\mathbb{T}^{n}$ can be found in \cite{[DHWX]}. In addition, one can also define the homogeneous Besov space with positive indices  in terms of finite differences. For the convenience of readers, we give the detail on periodic domain.   For $1\leq q\leq \infty$ and $0<\alpha<1$, the homogeneous Besov space $\dot{B}^{\alpha}_{q,\infty}(\mathbb{T}^{n})$ is the space of functions $f$ on the $n$-dimensional torus $\mathbb{T}^{n}=[0,1]^{n}$ equipped with the semi-norm
 \be\label{besov1} \|f\|_{\dot{B}^{\alpha}_{q,\infty}(\mathbb{T}^{n})}= \B\|\,|y|^{-\alpha}\B\| f(x-y)-f(x)\B\|_{L_{x}^{q}(\mathbb{T}^{n})}\B\|_{L_{y}^{\infty}(\mathbb{R}^{n})}<\infty,\ee
 and the nonhomogeneous Besov space $B^{\alpha}_{q,\infty}(\mathbb{T}^{n})$ is the set of functions $f\in L^{q}(\mathbb{T}^{n})$ equipped with the norm
 $$\ba
  \|f\|_{B^{\alpha}_{q,\infty}(\mathbb{T}^{n})}= \|f\|_{L^{q}(\mathbb{T}^{n})}
  +\|f\|_{\dot{B}^{\alpha}_{q,\infty}(\mathbb{T}^{n})}<\infty.\ea$$

In the usual manner, we will use $C$ to denote the absolute constants which may be different from line to line unless otherwise stated.

 \begin{definition}\label{eulerdefi}
	A vector field $v\in C_{\text{weak}}([0,T];L^{2}(\Omega)$ is called a weak solution of the incompressible homogeneous Euler equations with initial data $v_{0}\in L^{2}(\Omega)$ if $v$  satisfies \eqref{Euler} in the sense of distributions, namely
	\begin{enumerate}[(i)]
		\item  for any divergence free test function $\varphi\in C_{0}^{\infty}([0,T];C^{\infty}(\Omega))$, there holds
		$$
		\int_{\Omega}[v(x,T) \varphi(x,T)- v(x,0) \varphi(x,0)]\,dx =\int_{0}^{T}\int_{\Omega}v(x,t)(\partial_{t}\varphi(x,t)+v(x,t)\cdot\nabla\varphi(x,t))\,dxdt,
		$$
		\item[(ii)]
		$v$ is weakly divergence free, that is, for every test function $\psi\in C_{0}^{\infty}([0,T)\times\Omega)$,
		$$\int_0^T\int_{\Omega} v(x,t)\cdot\nabla\psi(x,t)\,dxdt=0.$$
		
	\end{enumerate}
\end{definition}

\begin{definition}
We call the pair $(\rho,v)$ a weak solution of the incompressible nonhomogeneous Euler equations with initial data
$\sqrt{\rho_0}v_{0}\in L^{2}(\Omega)$ if the equations \eqref{NEuler} are  satisfied in the sense of distributions and
\begin{equation}\label{inequa}
 \frac{1}{2}\int_{\Omega}\rho(x,t)|v(x,t)|^{2}\,dx\leq \frac{1}{2}\int_{\Omega}\rho(x,0)|v(x,0)|^{2}\,dx.
\end{equation}
\end{definition}

In the absence of density or vacuum, to establish the energy equality of \eqref{Euler} or \eqref{NEuler}, it suffices to use the mollifier in the spatial direction. From this, we are in the position to introduce its definition. Let $\eta$ be a non negative smooth function only supported in the space ball of radius 1 and its integral equals to 1. For $\varepsilon>0$, we define the rescaled mollifier $\eta_{\varepsilon}(x)=\frac{1}{\varepsilon^n}\eta(\frac{x}{\varepsilon})$. Naturally, for any function $f\in L^1_{\rm loc}(\Omega)$, we can define its mollified version by
$$f^\varepsilon(x)=(f*\eta_{\varepsilon})(x)=\int_{\Omega}f(x-y)\eta_{\varepsilon}(y)dy,\ \ x\in \Omega_\varepsilon,$$
where $\Omega\subset{\mathbb{R}}^{n}$ is a bounded domain and $\Omega_\varepsilon=\{x\in \Omega: d(x,\partial\Omega)>\varepsilon\}.$ By the way, in the proof of Theorem \ref{the1.1} and \ref{the1.2}, the following two lemmas on mollifying kernel will be frequently used. We list them below for for readers' convenience.
\begin{lemma} \label{lem2.2} (\cite{[WY]})  Let $1\leq p,q,p_1,p_2,q_1,q_2\leq \infty$  with $\frac{1}{p}=\frac{1}{p_1}+\frac{1}{p_2}$ and $\frac{1}{q}=\frac{1}{q_1}+\frac{1}{q_2}$. Assume $f\in L^{p_1}(0,T;W^{1,q_1}(\mathbb{T}^n))$ and $g\in L^{p_2}(0,T;L^{q_2}(\mathbb{T}^n))$, then for any $\varepsilon> 0$, there holds
		\begin{align} \label{fg'}
		\|(fg)^\varepsilon-f^\varepsilon g^\varepsilon\|_{L^p(0,T;L^q(\mathbb{T}^n))}\leq C\varepsilon \|f\|_{L^{p_1}(0,T;W^{1,q_1}( \mathbb{T}^n))}\|g\|_{L^{p_2}(0,T;L^{q_2}(\mathbb{T}^n))}.
		\end{align}
		Moreover, if $p_2,q_2<\infty$, then
		\begin{align}\label{limite'}
		\limsup_{\varepsilon \to 0}\varepsilon^{-1} \|(fg)^\varepsilon-f^\varepsilon g^\varepsilon\|_{L^p(0,T;L^q(\mathbb{T}^n))}=0.
		\end{align}
	\end{lemma}
\begin{lemma}(\cite{[NNT]})\label{lem2.1}
Suppose that $f\in L^{p}(0,T;L^{q}(\mathbb{T}^{n}))$, then for any $\varepsilon>0$, it holds
\be\label{}
\|\nabla f^{\varepsilon}\|_{L^{p}(0,T;L^{q}(\mathbb{T}^{n}))}
\leq C\varepsilon^{-1}\| f\|_{L^{p}(0,T;L^{q}(\mathbb{T}^{n}))},
\ee
and if $p,q<\infty$,
$$
\limsup_{\varepsilon\rightarrow0} \varepsilon\|\nabla f^{\varepsilon}\|_{L^{p}(0,T;L^{q}(\mathbb{T}^{n}))}=0.
$$
Furthermore, if $0<c_{1}\leq g\leq c_{2}<\infty$, then for any $\varepsilon>0$,
\be\ba
\B\|\nabla \B(\f{f^{\varepsilon}}{g^{\varepsilon}}\B)\B\|_{L^{p}(0,T;L^{q}(\mathbb{T}^{n}))}\leq C
\varepsilon^{-1}\|f\|_{L^{p}(0,T;L^{q}(\mathbb{T}^{n}))},
\ea\ee
and if $p,q<\infty$,
\be
\limsup_{\varepsilon\rightarrow0} \varepsilon\B\|\nabla \B(\f{f^{\varepsilon}}{g^{\varepsilon}}\B)\B\|_{L^{p}(0,T;L^{q}(\mathbb{T}^{n}))}
=0.\ee
\end{lemma}

When the density allows vacuum, different from before, we need the space-time mollifier in the process of establishing energy conservation. To this end, we first introduce a non-negative even smooth function $\eta$ supported in the space-time ball of radius 1 with its integral equalling to 1. On this basis, we define the rescaled space-time mollifier  $\eta_{\varepsilon}(t,x)=\frac{1}{\varepsilon^{n+1}}\eta(\frac{t}{\varepsilon},\frac{x}{\varepsilon})$ and then
$$
g^{\varepsilon}(t,x)=\int_{0}^{t}\int_{\Omega}g(s,y)\eta_{\varepsilon}(t-s,x-y)\,dyds
$$
for any function $g\in L^1_{\rm loc}([0,T)\times\Omega)$. Likewise, in the process to prove Theorem \ref{the1.3}, the following Lion's type commutators on space-time mollifying kernel will be used.
\begin{lemma}
	\label{pLions}Let $1\leq p,q,p_1,q_1,p_2,q_2\leq \infty$,  with $\frac{1}{p}=\frac{1}{p_1}+\frac{1}{p_2}$ and $\frac{1}{q}=\frac{1}{q_1}+\frac{1}{q_2}$.
	Let $\partial$ be a partial derivative in space or time, and in addition $\partial_t f,\ \nabla f \in L^{p_1}(0,T;L^{q_1}(\mathbb{T}^n))$, $g\in L^{p_2}(0,T;L^{q_2}(\mathbb{T}^n))$.   Then, there holds $$\|{\partial(fg)^\varepsilon}-\partial(f\,{g}^\varepsilon)\|_{L^p(0,T;L^q(\mathbb{T}^n))}\leq C\left(\|\partial_{t} f\|_{L^{p_{1}}(0,T;L^{q_{1}}(\mathbb{T}^n))}+\|\nabla f\|_{L^{p_{1}}(0,T;L^{q_{1}}(\mathbb{T}^n))}\right)\|g\|_{L^{p_{2}}(0,T;L^{q_{2}}(\mathbb{T}^n))},
	$$
	for some constant $C>0$ independent of $\varepsilon$, $f$ and $g$. Moreover, $${\partial{(fg)^\varepsilon}}-\partial{(f\,{g^\varepsilon})}\to 0\quad\text{ in } {L^{p}(0,T;L^{q}(\mathbb{T}^n))},$$
	as $\varepsilon\to 0$ if $p_2,q_2<\infty.$
\end{lemma}
The following  generalized Aubin-Lions Lemma  will be used to extend the energy equality up to the initial time.
\begin{lemma}[\cite{[Simon]}]\label{AL}
	Let $X\hookrightarrow B\hookrightarrow Y$ be three Banach spaces with compact imbedding $X \hookrightarrow\hookrightarrow Y$. Further, let there exist $0<\theta <1$ and $M>0$ such that
	\begin{equation}\label{le1}
	\|v\|_{B}\leq M\|v\|_{X}^{1-\theta}\|v\|_{Y}^\theta\ \ for\ all\ v\in X\cap Y.\end{equation}
Denote for $T>0$,
\begin{equation}\label{le2}
	W(0,T):=W^{s_0,r_0}((0,T), X)\cap W^{s_1,r_1}((0,T),Y)
\end{equation}
with
\begin{equation}\label{le3}
	\begin{aligned}
		&s_0,s_1 \in \mathbb{R}; \ r_0, r_1\in [1,\infty],\\
		s_\theta :=(1-\theta)s_0&+\theta s_1,\ \f{1}{r_\theta}:=\f{1-\theta}{r_0}+\f{\theta}{r_1},\ s^{*}:=s_\theta -\f{1}{r_\theta}.
	\end{aligned}
\end{equation}
Assume that $s_\theta>0$ and $F$ is a bounded set in $W(0,T)$. Then, we have

If $s_{*}\leq 0$, then $F$ is relatively compact in $L^p((0,T),B)$ for all $1\leq p< p^{*}:=-\f{1}{s^{*}}$.

If $s_{*}> 0$, then $F$ is relatively compact in $C((0,T),B)$.

\end{lemma}

\section{Incompressible homogeneous Euler equations}
In this section, we are concerned with energy conservation  class involving the  vorticity for the incompressible homogeneous Euler equations. For the periodic domains and whole spaces, we will use different methods to deal with them separately, i.e., the standard mollifier as \cite{[CET]} and Littlewood-Paley theory.

\begin{proof}[Proof of Theorem \ref{the1.1}]\no{\bf Periodic domain:} In this case, a key point is  to verify that $v\in L^{3}(0,T; L^{\f{3n}{n-1}}(\mathbb{T}^{n}))$. To this end, we notice that the Calder\'on-Zygmund theorem leads to
$$\|\nabla v\|_{L^{3}(0,T;L^{\f{3n}{n+2}}(\mathbb{T}^{n}))}\leq C\|\omega \|_{L^{3}(0,T;L^{\f{3n}{n+2}}(\mathbb{T}^{n}))},$$
which implies, after applying the triangle  inequality, Poincar\'e-Sobolev inequality and Gagliardo-Nirenberg inequality, that
$$\ba
\|v\|_{L^{\f{3n}{n-1}}(\mathbb{T}^{n})}\leq& \|v-\overline{v}\|_{L^{\f{3n}{n-1}}(\mathbb{T}^{n})}+   \| \overline{v}\|_{L^{\f{3n}{n-1}}(\mathbb{T}^{n})}\\
\leq& C\|\nabla v \|_{L^{\f{3n}{n+2}}(\mathbb{T}^{n})}+   C\| v\|_{L^{2}(\mathbb{T}^{n})}\\
\leq& C\|\omega \|_{L^{\f{3n}{n+2}}(\mathbb{T}^{n})}+   C\| v\|_{L^{2}(\mathbb{T}^{n})},
\ea
$$
where we have used $\|\nabla v \|_{L^{p}(\mathbb{T}^{n})}\leq C\|\omega \|_{L^{p}(\mathbb{T}^{n})}$ with $1<p<\infty$. Integrating the resulting inequality with respect to $t$ over $(0,T)$, it is clear that
\be\ba\label{key}
\|v\|_{L^{3}(0,T; L^{\f{3n}{n-1}}(\mathbb{T}^{n}))}
\leq& C\|\omega \|_{L^{3}(0,T;L^{\f{3n}{n+2}}(\mathbb{T}^{n}))}+   C\| v\|_{L^{\infty}(0,T;L^{2}(\mathbb{T}^{n}))}.
\ea
\ee

Subsequently, taking inner product of \eqref{Euler} with $(v^{\varepsilon})^{\varepsilon}$ and integrating over $(0,t)\times \mathbb{T}^{n}$, we infer that
\begin{equation}\label{euler1} \begin{aligned}
		\int_0^t\int_{\mathbb{T}^{n}}  \B[ ( v_{t})^{\varepsilon}+ \Div(  v\otimes v)^{\varepsilon}+\nabla \pi^\varepsilon \B]\cdot v^{\varepsilon}\,dxds=0,
\end{aligned}\end{equation}
where $v^\varepsilon$ means $v$ is mollified only in spatical direction. For the nonlinear term, the incompressible condition  div$\,v=0$ ensures
$$
\int_0^t\int_{\mathbb{T}^{n}}  \Div(  v\otimes v^{\varepsilon})\cdot v^{\varepsilon}\,dxds=0,~\text{and}~\int_0^t\int_{\mathbb{T}^n}\nabla \pi^\varepsilon\cdot v^\varepsilon dxds=0.
$$

As a result, we can rewrite \eqref{euler1}  as
\begin{equation}\label{reeuler}
	\begin{aligned}
	 \int_{\mathbb{T}^{n}}  {\frac{|v^{\varepsilon}(x,t)|^2}{2}}\,dx - \int_{\mathbb{T}^{n}} {\frac{|v(x,0)^{\varepsilon}|^2}{2}}\,dx
		=  \int_0^t\int_{\mathbb{T}^{n}}  [(  v\otimes v)^{\varepsilon}-v\otimes v^{\varepsilon}]\cdot  \nabla v^{\varepsilon}\,dx ds,
\end{aligned}\end{equation}
which concludes, after using H\"older inequality and triangle inequality, that
\begin{equation} \begin{aligned}\label{12}
 		&\left|\int_0^t\int_{\mathbb{T}^{n}}    [(  v\otimes v)^{\varepsilon}-v\otimes v^{\varepsilon}]\cdot \nabla v^{\varepsilon}\,dxds\right| \\
 		\leq &C\|\nabla v^\varepsilon\|_{L^{3}(0,T;L^{\f{3n}{n+2}}(\mathbb{T}^n))}\|( v\otimes v)^{\varepsilon}-v\otimes v^{\varepsilon}\|_{L^{\f{3}{2}}(0,T;L^{\f{3n}{2(n-1)}}(\mathbb{T}^n))}\\
 		\leq & C\|\nabla v\|_{L^{3}(0,T;L^{\f{3n}{n+2}}(\mathbb{T}^n))}\Big(\|(  v \otimes v)^\varepsilon-   v \otimes v\|_{L^{\f{3}{2}}(0,T;L^{\f{3n}{2(n-1)}}(\mathbb{T}^n))}\\
 &+\|  v\otimes v-   v\otimes v^\varepsilon\|_{L^{\f{3}{2}}(0,T;L^{\f{3n}{2(n-1)}}(\mathbb{T}^n))}\Big)\\
 		\leq & C\|\nabla v\|_{L^{3}(0,T;L^{\f{3n}{n+2}}(\mathbb{T}^n))} \Big(\|(  v \otimes v)^\varepsilon-   v \otimes v\|_{L^{\f{3}{2}}(0,T;L^{\f{3n}{2(n-1)}}(\mathbb{T}^n))}\\
 &+\| v\|_{L^3(L^{\f{3n}{n-1}}(\mathbb{T}^n))}\|v-v^\varepsilon\|_{L^3(0,T;L^{\f{3n}{n-1}}(\mathbb{T}^n))}\Big).
 \end{aligned}\end{equation}
 Thanks to \eqref{key} and standard properties of mollifier, we know that the right hand side of  \eqref{reeuler} becomes zero as $\varepsilon\rightarrow0$, which completes the proof of this case.

\no{\bf Whole space:}
In view of the Sobolev embedding and Calder\'on-Zygmund theorem, we have
$$\ba
\|v\|_{L^{\f{3n}{n-1}}(\mathbb{R}^{n})}
\leq  C\|\nabla v\|_{L^{\f{3n}{n+2}}(\mathbb{R}^{n})}\leq  C\|\omega \|_{L^{\f{3n}{n+2}}(\mathbb{R}^{n})}.
\ea
$$
According to the definition of weak solutions of the incompressible homogeneous Euler equations, it is clear that
\be\label{zheng1}
\f12\|S_{N}v\|_{L^{2}(\mathbb{R}^{n})}^{2} -\f12\|S_{N}v_{0}\|_{L^{2}(\mathbb{R}^{n})}^{2}=-\int_{0}^{t}\int_{\mathbb{R}^n}
\partial_{j}(v_{j}v_{i})S^{2}_{N}v_{i}\,dxds,
\ee
where $S_{N}v$ is defined in Section 2. To reformulate the right hand side of \eqref{zheng1}, as above, we derive from the divergence-free condition that
$$\int_{0}^{t}\int_{\mathbb{R}^n}
v_{j}\partial_{j}(S^{2}_{N}v_{i})^{2} \,dxds=0.$$
Consequently, we can rewrite the nonlinear term in \eqref{zheng1} as
\begin{equation}\label{s1}
\begin{split}
&-\int_{0}^{t}\int_{\mathbb{R}^n}
\partial_{j}(v_{j}v_{i})S^{2}_{N}v_{i}\,dxds\\
=&-\int_{0}^{t}\int_{\mathbb{R}^n}
\partial_{j}(v_{j}S^{2}_{N}v_{i})S^{2}_{N}v_{i}\,dxds-\int_{0}^{t}\int_{\mathbb{R}^n}
\partial_{j}[v_{j}(\text{I}_{\textnormal d}-S^{2}_{N})v_{i}]S^{2}_{N}v_{i}\,dxds\\
=&-\f12\int_{0}^{t}\int_{\mathbb{R}^n}
v_{j}\partial_{j}(S^{2}_{N}v_{i})^{2} \,dxds-\int_{0}^{t}\int_{\mathbb{R}^n}
\partial_{j}\big[v_{j}(\text{I}_{\textnormal d}-S^{2}_{N})v_{i}\big]S^{2}_{N}v_{i}\,dxds\\
=& \int_{0}^{t}\int_{\mathbb{R}^n}
S_{N}\big[v_{j}(\text{I}_{\textnormal d}-S^{2}_{N})v_{i}\big]S_{N}\partial_{j}v_{i}\,dxds,
\end{split}
\end{equation}
which arrives at, after using H\"older inequality,
\be\label{s2}\ba
&\left|\int_{0}^{t}\int_{\mathbb{R}^n}S_{N}\big[v_{j}(I-S^{2}_{N})v_{i}\big]S_{N}\partial_{j}v_{i}\,dxds\right| \\
\leq&\big\| S_{N}\big[v_{j}(\text{I}_{\textnormal d}-S^{2}_{N})v_{i}\big]\big\|_{L^{\f{3}{2}}(0,T;L^{\f{3n}{2(n-1)}}(\mathbb{R}^{n}))}  \big\| \nabla S_{N}v \big\|_{L^{3}(0,T;L^{\f{3n}{n+2}}(\mathbb{R}^{n}))}\\
\leq&  C\big\|  v_{j}(\text{I}_{\textnormal d}-S^{2}_{N})v_{i}\big\|_{L^{\f{3}{2}}(0,T;L^{\f{3n}{2(n-1)}}(\mathbb{R}^{n}))} \big\| S_{N}\partial_{j}v_{i} \big\|_{L^{3}(0,T;L^{\f{3n}{n+2}}(\mathbb{R}^{n}))}\\
\leq& C\|  v   \|_{L^3(0,T;L^{\f{3n}{n-1}}(\mathbb{R}^{n}))} \big\|\big(\text{I}_{\textnormal d}-S^2_{N}\big)v \big\|_{L^3(0,T;L^{\f{3n}{n-1}}(\mathbb{R}^{n}))}  \| \nabla  v \|_{L^{3}(0,T;L^{\f{3n}{n+2}}(\mathbb{R}^{n}))}.
\ea\ee
Thanks to Littlewood-Paley theory and Dominated Convergence Theorem, it follows that, as $N\rightarrow\infty$,
$$\big\|\big(\text{I}_{\textnormal d}-S^2_{N}\big)v \big\|_{L^3(0,T;L^{\f{3n}{n-1}}(\mathbb{R}^{n}))} \rightarrow0,
$$
which further implies, by \eqref{s1} and \eqref{s2}, that
$$
-\int_{0}^{t}\int_{\mathbb{R}^n}
\partial_{j}(v_{j}v_{i})S^{2}_{N}v_{i}\,dxds\rightarrow0.
$$
Finally, by taking $N\rightarrow\infty$ in \eqref{zheng1}, we conclude by \eqref{zheng1} that
$$ \|v(t)\|_{L^{2}(\mathbb{R}^{n})}^{2}= \|v_0\|_{L^{2}(\mathbb{R}^{n})}^{2},$$
which finish all the proof of this theorem.
\end{proof}

\section{ Incompressible nonhomogeneous Euler equations}

This section is devoted to establishing the energy conservation criterion for the incompressible nonhomogeneous Euler equations in terms of the vorticity.

\subsection{The case without vacuum}\label{sec4.1}
\begin{proof}[Proof of Theorem \ref{the1.2}]
Since $0<c_{1}\leq \rho\leq c_{2}<\infty$, similar with the model in Section 3, there still holds $v
\in L^{\infty}(0,T;L^{2}(\mathbb{T}^{n}))$ and $\nabla v\in L^{3}(0,T;L^{\f{3n}{n+2}}(\mathbb{T}^{n}))$. As the same manner of derivation of \eqref{key}, it is clear that \eqref{key01} implies $v\in L^{3}(0,T; L^{\f{3n}{n-1}}(\mathbb{T}^{n}))$. Thus, we can derive from H\"older inequality that
$v\in L^{3}(0,T;  L^{\f{3n}{n+2}}(\mathbb{T}^{n}))$. Combining this with $\nabla v\in L^{3}(0,T;L^{\f{3n}{n+2}}(\mathbb{T}^{n}))$, we obtain
\be\label{key2}
v\in L^{3}(0,T;W^{1,\f{3n}{n+2}}(\mathbb{T}^{n})).
\ee

Let $\phi(t)$ be a smooth function compactly supported in $(0,+\infty)$. We take inner product of $\eqref{NEuler}_2$ with $\left(\phi(t)\frac{(\rho v)^\varepsilon}{\rho^\varepsilon}\right)^\varepsilon$ and integrate it   over $(0,T)\times \mathbb{T}^d$ to obtain
\begin{equation}\label{nc1} \begin{aligned}
 \int_0^T\int_{\mathbb{T}^{n}} \B[\partial_{t}(\rho v)^{\varepsilon}+ \Div(\rho v\otimes v)^{\varepsilon}+\nabla \pi^\varepsilon\B]\cdot\phi(t)\f{(\rho v)^{\varepsilon}}{\rho^{\varepsilon}}\,dxdt=0,
 \end{aligned}\end{equation}
where $(\cdot)^\varepsilon$ means $(\cdot)$ is mollified only in spatical direction.

After several integration by parts, we arrive at
 \begin{equation}\label{c9}
 \begin{aligned}
& \f12\int_0^T\int_{\mathbb{T}^{n}}\phi(t)\partial_{t}\f{|(\rho v)^{\varepsilon}|^{2}}{\rho^{\varepsilon}}\,dxdt  \\
= &-\int_0^T\int_{\mathbb{T}^{n}}\phi(t) \B[\f{(\rho v)^{\varepsilon}}{\rho^{\varepsilon}} \B] \cdot\nabla\pi^{\varepsilon}\,dxdt
+\int_0^T\int_{\mathbb{T}^{n}}\phi(t) [(\rho v\otimes v)^{\varepsilon}-(\rho v)^{\varepsilon}\otimes v^{\varepsilon}]\cdot\nabla\B(\f{(\rho v)^{\varepsilon}}{\rho^{\varepsilon}} \B)\,dxdt\\& +\int_0^T\int_{\mathbb{T}^{n}}\phi(t)\B[ \rho^{\varepsilon}v^{\varepsilon}-(\rho v)^{\varepsilon} \B]\cdot\nabla\B(\f{(\rho v)^{\varepsilon}}{\rho^{\varepsilon}} \B)\cdot \f{(\rho v)^\varepsilon}
{\rho^{\varepsilon}}\,dxdt.\\
  \end{aligned}\end{equation}
 For the pressure term, the incompressible condition  div$\,v=0$ shows
 $$
 -\int_0^T\int_{\mathbb{T}^{n}} \phi(t)\frac{\rho^\varepsilon v^\varepsilon}{\rho^\varepsilon}\cdot\nabla\pi^\varepsilon \,dxdt=-\int_0^T\int_{\mathbb{T}^{n}}\phi(t) v^{\varepsilon} \cdot\nabla\pi^{\varepsilon}\,dxdt=0,
 $$
and hence
$$
-\int_0^T\int_{\mathbb{T}^{n}}\phi(t) \B[\f{(\rho v)^{\varepsilon}}{\rho^{\varepsilon}} \B] \cdot\nabla\pi^{\varepsilon}\,dxdt
=-\int_0^T\int_{\mathbb{T}^{n}}\phi(t) \B[\f{(\rho v)^{\varepsilon}-\rho^{\varepsilon}v^{\varepsilon}}{\rho^{\varepsilon}} \B] \cdot\nabla\pi^{\varepsilon}\,dxdt.
$$
Thanks to H\"older inequality, it is clear that
$$\ba
&\B|-\int_0^T\int_{\mathbb{T}^{n}}\phi(t) \B[\f{(\rho v)^{\varepsilon}-\rho^{\varepsilon}v^{\varepsilon}}{\rho^{\varepsilon}} \B] \cdot\nabla\pi^{\varepsilon}\,dxdt\B|\\
\leq& C\|\nabla\pi^{\varepsilon}\|_{L^{\f32}(0,T;L^{\f{3n}{2n-2}}(\mathbb{T}^{n}))} \|\f{(\rho v)^{\varepsilon}-\rho^{\varepsilon}v^{\varepsilon}}{\rho^{\varepsilon}}\|_{L^{3}(0,T;L^{\f{3n}{n+2}}(\mathbb{T}^{n}))}\\
\leq &  C  \|\nabla \pi^{\varepsilon}\|_{L^{\f32}(0,T;L^{\f{3n}{2n-2}}(\mathbb{T}^{n}))}       \| (\rho v)^{\varepsilon}-\rho^{\varepsilon}v^{\varepsilon}  \|_{L^{3}(0,T;L^{\f{3n}{n+2}}(\mathbb{T}^{n}))}.
\ea$$
Making use of Lemma \ref{lem2.2} and \ref{lem2.1}, we get
$$\ba
&\|(\rho v)^{\varepsilon}-\rho ^{\varepsilon}v^{\varepsilon}\|_{L^{3}(0,T;L^{\f{3n}{n+2}}(\mathbb{T}^{n}))}\leq \varepsilon\|v\|_{L^{3}(0,T;W^{1,\f{3n}{n+2}}(\mathbb{T}^{n}))}\|\rho  \|_{L^{\infty}(0,T;L^{\infty}(\mathbb{T}^{n}))},\\
&\limsup_{\varepsilon\rightarrow0}\varepsilon\|\nabla\pi^{\varepsilon}\|_{L^{\f32}(0,T;L^{\f{3n}{2n-2}}(\mathbb{T}^{n}))}  =0,
\ea$$
which implies
$$
  \limsup_{\varepsilon\rightarrow0}\B|-\int_0^T\int_{\mathbb{T}^{n}}\phi(t) \B[\f{(\rho v)^{\varepsilon}}{\rho^{\varepsilon}} \B] \cdot\nabla\pi^{\varepsilon}\,dxdt\B|=0.
$$

Now, we turn our attentions to the nonlinear term in \eqref{c9}.
Invoking  Lemma \ref{lem2.2} and \ref{lem2.1}, it follows that
\begin{equation}\label{nc17} \begin{aligned}
&\|(\rho v\otimes v)^{\varepsilon}-(\rho v)^{\varepsilon}\otimes v^{\varepsilon} \|_{L^{\f{3}{2}}(0,T;L^{\f{3n}{2n+1}}(\mathbb{T}^{n}))}\leq C\varepsilon\|v\|_{L^{3}(0,T;W^{1,\f{3n}{n+2}}(\mathbb{T}^{n})))}\|\rho v\|_{L^{3}(0,T;L^{\frac{3n}{n-1}}(\mathbb{T}^{n}))},\\
&\B\|\nabla\B(\f{(\rho v)^{\varepsilon}}{\rho^{\varepsilon}} \B)\B\|_{L^{3}(0,T;L^{\frac{3n}{n-1}}(\mathbb{T}^{n}))}\leq C\varepsilon^{-1}\|\rho v\|_{L^{3}(0,T;L^{\frac{3n}{n-1}}(\mathbb{T}^{n}))},\\ &\limsup_{\varepsilon\rightarrow0}\varepsilon\B\|\nabla\B(\f{(\rho v)^{\varepsilon}}{\rho^{\varepsilon}} \B)\B\|_{L^{3}(0,T;L^{\frac{3n}{n-1}}(\mathbb{T}^{n}))}=0.
\end{aligned}\end{equation}
Taking advantage of H\"older's inequality and Lemma \ref{lem2.2}, we have
  \begin{equation}\label{nc18}\begin{aligned}
 &\B|\int_0^T\int_{\mathbb{T}^{n}}\phi(t)[(\rho v\otimes v)^{\varepsilon}-(\rho v)^{\varepsilon}\otimes v^{\varepsilon}]\cdot\nabla\B(\f{(\rho v)^{\varepsilon}}{\rho^{\varepsilon}} \B)\,dxdt\B|\\
\leq & C\B\|\nabla\B(\f{(\rho v)^{\varepsilon}}{\rho^{\varepsilon}} \B)\B\|_{L^{3}(0,T;L^{\frac{3n}{n-1}}(\mathbb{T}^{n}))}\|(\rho v\otimes v)^{\varepsilon}-(\rho v)^{\varepsilon}\otimes v^{\varepsilon} \|_{L^{\f{3}{2}}(0,T;L^{\f{3n}{2n+1}}(\mathbb{T}^{n}))}\\
\leq & C\varepsilon\B\|\nabla\B(\f{(\rho v)^{\varepsilon}}{\rho^{\varepsilon}} \B)\B\|_{L^{3}(0,T;L^{\frac{3n}{n-1}}(\mathbb{T}^{n}))}\|v\|_{L^3(0,T;W^{1,\frac{3n}{n+2}}(\mathbb{T}^{n}))} \|\rho v\|_{L^{3}(0,T;L^{\frac{3n}{n-1}}(\mathbb{T}^{n}))}\\
\leq & C\varepsilon\B\|\nabla\B(\f{(\rho v)^{\varepsilon}}{\rho^{\varepsilon}} \B)\B\|_{L^{3}(0,T;L^{\frac{3n}{n-1}}(\mathbb{T}^{n}))}\|v\|_{L^3(0,T;W^{1,\frac{3n}{n+2}}(\mathbb{T}^{n}))} \| v\|_{L^{3}(0,T;L^{\frac{3n}{n-1}}(\mathbb{T}^{n}))}\\
\leq & C\varepsilon\B\|\nabla\B(\f{(\rho v)^{\varepsilon}}{\rho^{\varepsilon}} \B)\B\|_{L^{3}(0,T;L^{\frac{3n}{n-1}}(\mathbb{T}^{n}))},
 \end{aligned}\end{equation}
 which follows from \eqref{nc17} that
$$
\limsup_{\varepsilon\rightarrow0}\B|\int_0^T\int_{\mathbb{T}^{n}}\phi(t)[(\rho v\otimes v)^{\varepsilon}-(\rho v)^{\varepsilon}\otimes v^{\varepsilon}]\cdot\nabla\B(\f{(\rho v)^{\varepsilon}}{\rho^{\varepsilon}} \B)\,dxdt\B|=0.$$

Thus, it suffices to deal with the term $\int_0^T\int_{\mathbb{T}^{n}}\phi(t)\B[ \rho^{\varepsilon}v^{\varepsilon}-(\rho v)^{\varepsilon} \B]\cdot \nabla\f{(\rho v)^\varepsilon}{\rho^{\varepsilon}}\cdot\f{(\rho v)^\varepsilon}
{\rho^{\varepsilon}}\,dxdt$. From Lemma \ref{lem2.1} and $\rho v\in  L^{3}(0,T;L^{\frac{3n}{n-1}}(\mathbb{T}^{n}))$, there holds
\begin{equation}\label{1.3.15}
\limsup_{\varepsilon\rightarrow0}\varepsilon\B\|\nabla \left(\f{(\rho v)^\varepsilon}{\rho^{\varepsilon}}\right)\B\|_{ L^{3}(0,T;L^{\frac{3n}{n-1}}(\mathbb{T}^{n})) }
=0.
\end{equation}
In addition, from Lemma \ref{lem2.2}, it follows that
\begin{equation}\label{1.3.16}
\|\rho^{\varepsilon}v^{\varepsilon}-(\rho v)^{\varepsilon}\|_{L^{3}(0,T;L^{\f{3n}{n+2}}(\mathbb{T}^{n}))} \leq C\varepsilon\|v\|_{L^{3}(0,T;W^{1,\f{3n}{n+2}}(\mathbb{T}^{n}))}  \|\rho\|_{L^{\infty}(0,T;L^{\infty}(\mathbb{T}^{n}))}.
\end{equation}
Combining H\"older's inequality  and \eqref{1.3.16}, we derive
\begin{equation}\label{1.3.17} \begin{aligned}
&\B|\int_0^T\int_{\mathbb{T}^{n}}\phi(t)\B[ \rho^{\varepsilon}v^{\varepsilon}-(\rho v)^{\varepsilon} \B]\cdot\nabla \left(\f{(\rho v)^\varepsilon}{\rho^{\varepsilon}}\right)\cdot \f{(\rho v)^\varepsilon}
{\rho^{\varepsilon}}\,dxdt\B| \\
\leq&C  \|\rho^{\varepsilon}v^{\varepsilon}-(\rho v)^{\varepsilon}\|_{L^{3}(0,T;L^{\f{3n}{n+2}}(\mathbb{T}^{n}))}\B\|\f{(\rho v)^\varepsilon}
{\rho^{\varepsilon}}\B\|_{ L^{3}(0,T;L^{\frac{3n}{n-1}}(\mathbb{T}^{n})) }\B\|\nabla\left(\f{(\rho v)^\varepsilon}{\rho^{\varepsilon}}\right)\B\|_{ L^{3}(0,T;L^{\frac{3n}{n-1}}(\mathbb{T}^{n})) } \\
\leq& C\varepsilon \|v\|_{L^{3}(0,T;W^{1,\f{3n}{n+2}}(\mathbb{T}^{n}))} \|\rho\|_{L^{\infty}(0,T;L^{\infty}(\mathbb{T}^{n}))} \|v\|_{ L^{3}(0,T;L^{\frac{3n}{n-1}}(\mathbb{T}^{n})) }\B\|\nabla\left(\f{(\rho v)^\varepsilon}{\rho^{\varepsilon}}\right)\B
\|_{ L^{3}(0,T;L^{\frac{3n}{n-1}}(\mathbb{T}^{n})) }\\
\leq& C\varepsilon \B\|\nabla\left(\f{(\rho v)^\varepsilon}{\rho^{\varepsilon}}\right)\B
\|_{ L^{3}(0,T;L^{\frac{3n}{n-1}}(\mathbb{T}^{n})) },\\
\end{aligned}\end{equation}
which together with \eqref{1.3.15} implies that
$$
\limsup_{\varepsilon\rightarrow0}
\B|\int_0^T\int_{\mathbb{T}^{n}}\phi(t)\B[ \rho^{\varepsilon}v^{\varepsilon}-(\rho v)^{\varepsilon} \B] \cdot\nabla \left(\f{(\rho v)^\varepsilon}{\rho^{\varepsilon}}\right)\cdot \f{(\rho v)^\varepsilon}
{\rho^{\varepsilon}}\,dxdt\B|=0.
$$
Putting the above estimates together, using integration by parts with respect to $t$ and letting $\varepsilon\rightarrow 0$, for any $\phi(t)\in \mathcal{D} (0,T) $, we get
\begin{equation}\label{c27}
	\begin{aligned}
		&-\int_0^T\int_{\mathbb{T}^{n}}\phi(t)_{t}
		\frac{1}{2}\rho |v|^2\,dxdt=0.\\
\end{aligned}\end{equation}

With \eqref{c27} at hand, we have proved the energy conservation in the sense of distributions of time. To finish all the proof, it suffices to establish the energy equality up to the initial time $t=0$ by the similar method in \cite{[CLWX]} and \cite{[Yu2]}, for the convenience of readers and integrity of this paper, we provide the details here.
First we prove the continuity of $\sqrt{\rho}v(t)$ in the strong topology as $t\to 0^+$.  To this end, we define the function $f$ on $[0,T]$ as
$$f(t)=\int_{\mathbb{T}^n}(\rho v)(t,x)\cdot \varphi(x) \,dx,\ {\rm for\, any \,vector\,field}\ \varphi(x)\in \mathcal{D}(\mathbb{T}^n) ~\text{with}~\text{div}\,\varphi=0,$$
which is a continuous function with respect to $t\in [0,T]$. Moreover, because
$$\rho \in L^\infty(0,T; L^\infty (\mathbb{T}^n))\ {\rm and} \ \sqrt{\rho}v\in L^\infty(0,T;L^2(\mathbb{T}^n)),$$
it holds $\rho v\in L^\infty(0,T;L^{2}(\mathbb{T}^n)).$ Thus, according to the moment equation, for any vector field $\varphi\in \mathcal{D}(\mathbb{T}^n)$, we have
$$\frac{d}{dt}\int_{\mathbb{T}^n} (\rho v)(t,x)\cdot \varphi(x)\,dx=\int_{\mathbb{T}^n}\rho v\otimes v:\nabla \varphi(x)\,dx.$$
Then it follows from the Corollary 2.1 in  \cite{[Feireisl2004]} that
\begin{equation}\label{c28}
	\rho v\in C([0,T];L^{2}_{\text{weak}}(\mathbb{T}^n)).
\end{equation}
Similarly, we also have
\begin{equation}\label{c29}
	\sqrt{\rho}\in C([0,T];L^{l}_{\text{weak}}(\mathbb{T}^n)),\ \text{for\ any}\ l\in [1,+\infty).
\end{equation}
Based on above facts, we claim that for any $t_0\geq 0$,
$$\lim\limits_{t\rightarrow t_0^+}\|\sqrt{\rho} v(t)\|_{L^2(\mathbb{T}^n)}=\|\sqrt{\rho}v(t_0)\|_{L^2(\mathbb{T}^n)}.$$

In fact, by the
energy inequality \eqref{inequa} and \eqref{key01},  we have
\begin{equation}\label{d17}
	\begin{aligned}
		0&\leq \overline{\lim_{t\rightarrow 0^+}}\int_{\mathbb{T}^{n}} |\sqrt{\rho} v-\sqrt{\rho_0}v_0|^2 dx\\
		&=2\overline{\lim_{t\rightarrow 0^+}}\left(\int_{\mathbb{T}^{n}} \f{1}{2}\rho |v|^2 dx-\int\f{1}{2}\rho_0 |v_0|^2dx\right)+2\overline{\lim_{t\rightarrow 0^+}}\int_{\mathbb{T}^{n}}\sqrt{\rho_0}v_0\left(\sqrt{\rho_0}v_0-\sqrt{\rho} v\right)dx\\
		&\leq 2\overline{\lim_{t\rightarrow 0^+}}\int_{\mathbb{T}^{n}} \sqrt{\rho_0}v_0\left(\sqrt{\rho_0}v_0-\sqrt{\rho}v\right)dx\\
		&=0.
	\end{aligned}
\end{equation}
Regarding the last equality sign, it comes from
\begin{equation}\label{3.25}
	\begin{aligned}
		&2\overline{\lim_{t\rightarrow 0^+}}\int_{\mathbb{T}^{n}} \sqrt{\rho_0}v_0\left(\sqrt{\rho_0}v_0-\sqrt{\rho}v\right)dx\\
		=&2\overline{\lim_{t\rightarrow 0^+}}\int_{\mathbb{T}^{n}} \frac{\sqrt{\rho_0}v_0}{\sqrt{\rho}}\left(\sqrt{\rho}\sqrt{\rho_0}v_0-\rho v\right)dx\\
		\leq &2\overline{\lim_{t\rightarrow 0^+}}\int_{\mathbb{T}^{n}} \frac{\sqrt{\rho_0}v_0}{\sqrt{\rho}}\left(\sqrt{\rho}\sqrt{\rho_0}v_0-\rho_0 v_0\right)dx+2\overline{\lim_{t\rightarrow 0^+}}\int_{\mathbb{T}^{n}} \frac{\sqrt{\rho_0}v_0}{\sqrt{\rho}}\left({\rho_0}v_0-\rho v\right)dx\\
		\leq &2\overline{\lim_{t\rightarrow 0^+}}\int_{\mathbb{T}^{n}} (\sqrt{\rho_0}v_0)^2\left(\sqrt{\rho}-\sqrt{\rho_0}\right)dx+2\overline{\lim_{t\rightarrow 0^+}}\int_{\mathbb{T}^{n}}\sqrt{\rho_0} v_0\left({\rho_0}v_0-\rho v\right)dx\\
		=&0,
	\end{aligned}
\end{equation}
where we used \eqref{c28}, \eqref{c29} and $\sqrt{\rho_0}v_0\in L^{2+\delta}$ for any $\delta>0$. Until now, we have proved
\begin{equation}\label{d18}
	\sqrt{\rho} v(t)\rightarrow \sqrt{\rho }v(0)\ \ {\rm strongly\ in}\ L^2(\mathbb{T}^n)\ as\ t\rightarrow 0^+.
\end{equation}
Similarly, one can derive the right temporal continuity  of $\sqrt{\rho}v$ in $L^2(\mathbb{T}^n)$, i.e., for any $t_0\geq 0$,
\begin{equation}\label{d19}
	\sqrt{\rho} v(t)\rightarrow \sqrt{\rho }v(t_0)\ \ {\rm strongly\ in}\ L^2(\mathbb{T}^n)\ as\ t\rightarrow t_0^+.
\end{equation}
Before we go any further, it should be noted that \eqref{c27} remains valid for any function $\phi$ belonging to $W^{1,\infty}$ rather than $C^1$. From this, for any $t_0>0$, we redefine the test function $\phi$ as $\phi_\tau$ for some positive $\tau$ and $\alpha $ such that $\tau +\alpha <t_0$, namely,
\begin{equation}\label{4.16}
	\phi_\tau(t)=\left\{\begin{array}{lll}
		0, & 0\leq t\leq \tau,\\
		\f{t-\tau}{\alpha}, & \tau\leq t\leq \tau+\alpha,\\
		1, &\tau+\alpha \leq t\leq t_0,\\
		\f{t_0-t}{\alpha }, & t_0\leq t\leq t_0 +\alpha ,\\
		0, & t_0+\alpha \leq t.
	\end{array}\right.
\end{equation}
Then substituting this test function into \eqref{c27}, it yields
\begin{equation}
	\begin{aligned}
		-\f{1}{\alpha}\int_\tau^{\tau+\alpha}\int_{\mathbb{T}^{n}}& \f{1}{2}\rho |v|^2\,dxdt+\f{1}{\alpha}\int_{t_0}^{t_0+\alpha}\int_{\mathbb{T}^{n}} 	\f{1}{2}\rho |v|^2\,dxdt=0.
	\end{aligned}	
\end{equation}
Taking $\alpha\rightarrow 0$ and using the right continuity of $\sqrt{\rho} v(t)$ in $L^2(\mathbb{T}^n)$, one has
\begin{equation}\label{4.18}
	\begin{aligned}
		-\int_{\mathbb{T}^{n}}&\f{1}{2}\rho |v|^2(\tau)\,dx+\int_{\mathbb{T}^{n}}\f{1}{2}\rho |v|^2(t_0)\,dx=0.
	\end{aligned}
\end{equation}
Finally, letting $\tau\rightarrow 0$ and using \eqref{d19}, we can obtain
\begin{equation}\label{4.19}\ba
	\int_{\mathbb{T}^{n}}\f{1}{2}\rho |v|^2(t_0)\,dx=&\int_{\mathbb{T}^{n}}\f{1}{2}\rho_0 |v_0|^2\,dx,
	\ea\end{equation}
which completes the proof of Theorem \ref{the1.2}.
\end{proof}
\subsection{The case with vacuum}
In the sprit of \cite{[YWW]}, we show the key lemma to treat the case with vacuum.
\begin{lemma}\label{lem4.1}
Assume that $0\leq\rho<+\infty$ and  $\int_{\mathbb{T}^n} \rho dx>0$, if in addition suppose  $\omega\in L^{3}(0,T;L^{\f{3n}{n+2}}(\mathbb{T}^{n}))$ and $\sqrt{\rho v}\in L^{\infty}(0,T;L^{2}(\mathbb{T}^{n}))$, then there holds
$$
v\in L^{3}(0,T;L^{\f{3n}{n-1}}(\mathbb{T}^{n})).
$$
\end{lemma}
\begin{proof}
With the help of triangle inequality and Poincar\'e-Sobolev inequality, it is clear that
\be\ba\label{4.24}
\|v\|_{L^{3}(0,T;L^{\f{3n}{n-1}}(\mathbb{T}^{n}))}
\leq&\B\|v-\f{1}{|\mathbb{T}^n|}\int_{\mathbb{T}^n}v\,dy
\B\|_{L^{3}(0,T;L^{\f{3n}{n-1}}(\mathbb{T}^{n}))}+\B\|\f{1}{|\mathbb{T}^n|}\int_{\mathbb{T}^n}v\,dy
\B\|_{L^{3}(0,T;L^{\f{3n}{n-1}}(\mathbb{T}^{n}))}\\
\leq& C\|\nabla v
\|_{L^{3}(0,T;L^{\f{3n}{n+2}}(\mathbb{T}^{n}))}+C\B\| \int_{\mathbb{T}^n}v\,dy
\B\|_{L^{3} (0,T)}.
\ea\ee
Thus, to prove this lemma, it is enough to verify
$
\int_{\mathbb{T}^n}v\,dy\in L^{3}(0,T).
$
For this purpose, we apply H\"older inequality and classical Poincar\'e inequality to obtain
\be\ba\label{3.28}
\B|\int_{\mathbb{T}^n}\rho\left(v-\f{1}{|\mathbb{T}^n|}\int_{\mathbb{T}^n}v\,dy\right)  dx\B|\leq& C \|\rho\|_{L^\f{3n}{2n-2}(\mathbb{T}^{n})}
\B\|v-\f{1}{|\mathbb{T}^n |}\int_{\mathbb{T}^n}v\,dy\B\|_{L^{\f{3n}{n+2}}(\mathbb{T}^{n})}\\
\leq& C\|\rho\|_{L^\f{3n}{2n-2}(\mathbb{T}^{n})}
\|\nabla v\|_{L^{\f{3n}{n+2}}(\mathbb{T}^{n})}.
\ea\ee
It follows from the H\"older's inequality once again and  the upper bound of density, that
\be\label{3.29}
|\int_{\mathbb{T}^{n}} \rho v\,dx|\leq C\|\sqrt{\rho}v\|_{L^{2}(\mathbb{T}^{n})}.
\ee
Making use of triangle inequality once again, \eqref{3.28} and \eqref{3.29}, we conclude
$$\ba
\B|\int_{\mathbb{T}^n} v\,dy\B|=&\f{|\mathbb{T}^n|}{\int_{\mathbb{T}^d} \rho\,dx} \B|\int_{\mathbb{T}^n} \rho\left(\f{1}{|\mathbb{T}^n|}\int_{\mathbb{T}^n}v\,dy\right) \,dx\B|\\
\leq&\f{C}{\int_{\mathbb{T}^n} \rho\,dx} \B|\int_{\mathbb{T}^n}\rho\B[\left( \f{1}{|\mathbb{T}^n|}\int_{\mathbb{T}^n}v\,dy\right)-v\B] \,dx\B|+
\f{C}{\int_{\mathbb{T}^n} \rho\,dx} \B| \int_{\mathbb{T}^n} \rho v\,dx\B|\\
\leq& C
\|\nabla v\|_{L^{\f{3n}{n+2}}(\mathbb{T}^{n})}+C\|\sqrt{\rho}v\|_{L^{2}(\mathbb{T}^{n})}.
\ea$$
which implies, after integrating in time, that
$$\ba
\B\|\int_{\mathbb{T}^n} v \,dy\B\|_{L^{3}(0,T)}
\leq& C
\|\nabla v\|_{L^{3}(0,T;L^{\f{3n}{n+2}}(\mathbb{T}^{n}))}+C\|\sqrt{\rho}v\|_{L^{\infty}(0,T;L^{2}(\mathbb{T}^{n}))}.
\ea$$
Inserting this inequality into \eqref{4.24}, we find that
$$
\|v\|_{L^{3}(0,T;L^{\f{3n}{n-1}}(\mathbb{T}^{n}))}
\leq C\|\nabla v\|_{L^{3}(0,T;L^{\f{3n}{n+2}}(\mathbb{T}^{n}))}+C\|\sqrt{\rho}v\|_{L^{\infty}(0,T;L^{2}(\mathbb{T}^{n}))}.
$$
As a result, the divgence-free condition and Calder\'on-Zygmund theorem further imply
$$
\|v\|_{L^{3}(0,T;L^{\f{3n}{n-1}}(\mathbb{T}^{n}))}
\leq C\|\omega\|_{L^{3}(0,T;L^{\f{3n}{n+2}}(\mathbb{T}^{n}))}+C\|\sqrt{\rho}v\|_{L^{\infty}(0,T;L^{2}(\mathbb{T}^{n}))}.
$$
Thus, the proof of this lemma is finished.
\end{proof}

\begin{proof}[Proof of Theorem \ref{the1.3}]
When the density may contain vacuum, then $\frac{(\rho v)^\varepsilon}{\rho^\varepsilon}$ with mollifier kernal just in spatial direction is no longer a suitable test function. To avoid this difficulty, we intend to mollify the velocity $v$ both in space and time direction. Let $\phi(t)$ be a smooth function compactly supported in $(0,+\infty)$.
Taking inner product of $\eqref{NEuler}_2$ with $(\phi v^{\varepsilon})^\varepsilon$, using the incompressible condition, and integrating over $(0,T)\times \mathbb{T}^n$, it yields that
\begin{equation}\label{c1} \begin{aligned}
		\int_0^T\int_{\mathbb{T}^{n}} \phi(t)\B[\partial_{t}(\rho v)^{\varepsilon}+ \Div(\rho v\otimes v)^{\varepsilon} +\nabla \pi^\varepsilon\B]\cdot v^{\varepsilon}\,dxdt=0.
\end{aligned}\end{equation}
On the basis of \eqref{c1}, we can deduce from a straightforward computation that
\begin{equation}\label{c2}
	\begin{aligned}
		&\int_0^T\int_{\mathbb{T}^{n}} \phi(t)\partial_{t} (\rho v )^{\varepsilon}\cdot v^{\varepsilon}\,dxdt\\
=&\int_0^T\int_{\mathbb{T}^{n}} \phi(t)\B[ \partial_{t} (\rho v )^{\varepsilon}-\partial_{t}(\rho v^{\varepsilon})\B]\cdot v^{\varepsilon}\,dxdt+ \int_0^T\int_{\mathbb{T}^{n}} \phi(t)\partial_{t}(\rho v^{\varepsilon})\cdot v^{\varepsilon}\,dxdt \\
		=& \int_0^T\int_{\mathbb{T}^{n}} \phi(t)\B[\partial_{t} (\rho v )^{\varepsilon}-\partial_{t}(\rho v^{\varepsilon})\B]\cdot v^{\varepsilon}\,dxdt+\int_0^T\int_{\mathbb{T}^{n}} \phi(t)\rho\partial_t{\frac{|v^{\varepsilon}|^2}{2}}\,dxdt \\
		&+\int_0^T\int_{\mathbb{T}^{n}} \phi(t)\rho_{t}|v^{\varepsilon}|^2\,dxdt.
\end{aligned}\end{equation}

Thanks to integration by parts and the mass equation $\eqref{NEuler}_1$, we can reformulate the nonlinear term as
\begin{align}
	&\int_0^T\int_{\mathbb{T}^{n}}\phi(t) \Div(\rho v\otimes v)^{\varepsilon}\cdot v^{\varepsilon}\,dxdt\nonumber\\
	=& \int_0^T\int_{\mathbb{T}^{n}}\phi(t)  \Div[(\rho v\otimes v)^{\varepsilon}-(\rho  v)\otimes v^{\varepsilon}]\cdot v^{\varepsilon}\,dxdt+\int_0^T\int_{\mathbb{T}^{n}}\phi(t)\Div(\rho  v\otimes v^{\varepsilon})\cdot v^{\varepsilon}\,dxdt\nonumber\\
	=& -\int_0^T\int_{\mathbb{T}^{n}}\phi(t)  [(\rho v\otimes v)^{\varepsilon}-(\rho  v)\otimes v^{\varepsilon}]\cdot\nabla v^{\varepsilon} \,dxdt+  \int_0^T\int_{\mathbb{T}^{n}} \phi(t)\Big(\Div (\rho v ) |v^{\varepsilon}|^{2}+\f12 \rho v\cdot \nabla|v^{\varepsilon}  |^{2}
	\Big)\,dxdt\nonumber\\
	=& -\int_0^T\int_{\mathbb{T}^{n}}\phi(t) [(\rho v\otimes v)^{\varepsilon}-(\rho  v)\otimes v^{\varepsilon}]\cdot\nabla  v^{\varepsilon}  \,dxdt+\f{1}{2}\int_0^T\int_{\mathbb{T}^{n}}\phi(t) \Div (\rho v ) |v^{\varepsilon}|^{2}\,dxdt\nonumber\\
	=& -\int_0^T\int_{\mathbb{T}^{n}}\phi(t)  [(\rho v\otimes v)^{\varepsilon}-(\rho  v)\otimes v^{\varepsilon}]\cdot\nabla v^{\varepsilon}  \,dxdt-\frac{1}{2}\int_0^T\int_{\mathbb{T}^{n}} \phi(t) \rho_{t} |v^{\varepsilon}|^{2}\,dxdt. \label{c3}
\end{align}
For the pressure term, due to the divergence free condition $\text{div}\,v=0$, one has
\begin{equation}
	\int_0^T\int_{\mathbb{T}^{n}} \phi(t) v^\varepsilon \cdot\nabla \pi^\varepsilon\,dxdt=-\int_0^T\int_{\mathbb{T}^{n}} \phi(t)\text{div}\, v^\varepsilon \pi^\varepsilon\,dxdt=0.
\end{equation}
As a consequence, collecting the above estimates and using integration by parts, we arrive at the desired equality
\begin{equation}\label{c91}
	\begin{aligned}
		&-\int_0^T\int_{\mathbb{T}^{n}} \phi(t)_t\left(\rho{\frac{|v^{\varepsilon}|^2}{2}} \right)\,dxdt  \\
		=&-\int_0^T\int_{\mathbb{T}^{n}} \phi(t)v^{\varepsilon} \B[\partial_{t} (\rho v )^{\varepsilon}-\partial_{t}(\rho v^{\varepsilon})\B]\,dxdt\\&+\int_0^T\int_{\mathbb{T}^{n}}\phi(t) [(\rho v\otimes v)^{\varepsilon}-(\rho  v)\otimes v^{\varepsilon}]\cdot\nabla v^{\varepsilon} \,dxdt.
\end{aligned}\end{equation}
For the first term in \eqref{c91}, by H\"older inequality and Lemma \ref{pLions}, there holds
\begin{equation}\label{3.81}\begin{aligned}
&\int_0^T\int_{\mathbb{T}^{n}} \phi(t)\B[\partial_{t} (\rho v )^{\varepsilon}-\partial_{t}(\rho v^{\varepsilon})\B]\cdot v^{\varepsilon} \,dxdt\\
&\leq C\|v^{\varepsilon}\|_{L^3(0,T;L^{\f{3n}{n-1}}(\mathbb{T}^{n}))}
\|\partial_{t} (\rho v )^{\varepsilon}-\partial_{t}(\rho v^{\varepsilon})\|_{L^{\f{3}{2}}(0,T;L^{\f{3n}{2n+1}}(\mathbb{T}^{n}))}\\
&\leq C\|v\|^{2}_{L^3(0,T;L^{\f{3n}{n-1}}(\mathbb{T}^{n}))}
\left(\|\rho_{t}\|_{L^{3}(0,T;L^{\f{3n}{n+2}}(\mathbb{T}^{n}))}+\|\nabla \rho\|_{L^{3}(0,T;L^{\f{3n}{n+2}}(\mathbb{T}^{n}))}\right).
\end{aligned}\end{equation}
To control the terms on the right hand side of \eqref{3.81}, we apply the mass equation $\eqref{NEuler}_1$ as follows
$$
\rho_t=-2\sqrt{\rho}v\cdot\nabla\sqrt{\rho} \ \text{and}\ \nabla \rho=2\sqrt{\rho}\nabla \sqrt{\rho}.$$
And then the triangle inequality and H\"older's inequality allow us to get
\begin{equation}\label{3.91}\begin{aligned}
\|\rho_{t}\|_{L^{3}(0,T;L^{\f{3n}{n+2}}(\mathbb{T}^{n}))}&\leq C\|-2\sqrt{\rho}v\cdot\nabla\sqrt{\rho} \|_{L^{3}(0,T;L^{\f{3n}{n+2}}(\mathbb{T}^{n}))}\\
&\leq C\|v \|_{L^3(0,T;L^{\f{3n}{n-1}}(\mathbb{T}^{n}))}\|\nabla\sqrt{\rho} \|_{L^{\infty}(0,T;L^{n}(\mathbb{T}^{n}))},
 \end{aligned}\end{equation}
and
\begin{equation}\label{3.92}
	\begin{aligned}
	\|\nabla \rho\|_{L^{3}(0,T;L^{\f{3n}{n+2}}(\mathbb{T}^{n}))}\leq & C\|\sqrt{\rho}\nabla \sqrt{\rho}\|_{L^{3}(0,T;L^{\f{3n}{n+2}}(\mathbb{T}^{n}))}\leq C\|\nabla \sqrt{\rho}\|_{L^{\infty}(0,T;L^{n}(\mathbb{T}^{n}))}.
	\end{aligned}
\end{equation}
Substituting
  \eqref{3.91} and \eqref{3.92} into \eqref{3.81}, we have
 \begin{equation}\label{c3.11}\begin{aligned}
&\int_0^T\int_{\mathbb{T}^{n}} \phi(t)v^{\varepsilon} \B[\partial_{t} (\rho v )^{\varepsilon}-\partial_{t}(\rho v^{\varepsilon})\B]\,dxds\\\leq&
C\|v\|^{2}_{L^3(0,T;L^{\f{3n}{n-1}}(\mathbb{T}^{n}))}
\Big(\|v \|_{L^3(0,T;L^{\f{3n}{n-1}}(\mathbb{T}^{n}))}+1\Big)\|\nabla\sqrt{\rho} \|_{L^{\infty}(0,T;L^{n}(\mathbb{T}^{n}))}\\
&+C\|v\|^{2}_{L^3(0,T;L^{\f{3n}{n-1}}(\mathbb{T}^{n}))}\| \nabla v\|_{L^{3}(0,T;L^{\f{3n}{n+2}}(\mathbb{T}^{n}))}\\
\leq& C.
 \end{aligned}\end{equation}
Making use of  Lemma \ref{pLions}, by choosing $p_1=3, q_1=\frac{3n}{n+2}, p_2=3$ and $q_2=\f{3n}{n-1}$, we see that, as $\varepsilon\rightarrow0$,
 $$
 \int_0^T\int_{\mathbb{T}^{n}} \phi(t)v^{\varepsilon} \B[\partial_{t} (\rho v )^{\varepsilon}-\partial_{t}(\rho v^{\varepsilon})\B]\,dxds\rightarrow0.$$
Using the integration by parts and H\"older inequality, we can estimate the second term on the right hand side of \eqref{c91} as follows
 \begin{equation}\label{3.11}\begin{aligned}
 		&\left|\int_0^T\int_{\mathbb{T}^{n}}\phi(t)  [(\rho v\otimes v)^{\varepsilon}-(\rho  v)\otimes v^{\varepsilon}]\cdot\nabla v^{\varepsilon}  \,dxdt\right| \\
 		\leq &C\|\nabla v^\varepsilon\|_{L^{3}(0,T;L^{\f{3n}{n+2}}(\mathbb{T}^{n}))}\|(\rho v\otimes v)^{\varepsilon}-(\rho  v)\otimes v^{\varepsilon}\|_{L^{\f{3}{2}}(0,T; L^{\f{3n}{2(n-1)}}(\mathbb{T}^{n}))}\\
 		\leq & C\|\nabla v\|_{L^{3}(0,T;L^{\f{3n}{n+2}}(\mathbb{T}^{n}))}\|(\rho v \otimes v)^\varepsilon- \rho v \otimes v\|_{L^{\f{3}{2}}(0,T:L^{\f{3n}{2(n-1)}}(\mathbb{T}^{n}))}\\
 &+C\|\nabla v\|_{L^{3}(0,T;L^{\f{3n}{n+2}}(\mathbb{T}^{n}))}\|\rho v\otimes v- \rho v\otimes v^\varepsilon\|_{L^{\f{3}{2}}(0,T;L^{\f{3n}{2(n-1)}}(\mathbb{T}^{n}))}\\
 		\leq & C\|\nabla v\|_{L^{3}(0,T;L^{\f{3n}{n+2}}(\mathbb{T}^{n}))}\|(\rho v \otimes v)^\varepsilon- \rho v \otimes v\|_{L^{\f{3}{2}}(0,T;L^{\f{3n}{2(n-1)}}(\mathbb{T}^{n}))}\\
 &+C\|\nabla v\|_{L^{3}(0,T;L^{\f{3n}{n+2}}(\mathbb{T}^{n}))}\|\rho v\|_{L^3(0,T;L^{\f{3n}{n-1}}(\mathbb{T}^{n}))}\|v-v^\varepsilon\|_{L^3(0,T;L^{\f{3n}{n-1}}(\mathbb{T}^{n}))}.
 \end{aligned}\end{equation}
 According to the standard properties of mollifiers, it yields that
\begin{equation}\label{4.32}
	\int_0^T\int_{\mathbb{T}^{n}}\phi(t)  [(\rho v\otimes v)^{\varepsilon}-(\rho  v)\otimes v^{\varepsilon}]\cdot\nabla v^{\varepsilon}  \,dxdt \rightarrow0,\  \text{ as }\varepsilon\rightarrow0.
\end{equation}
Combining \eqref{c91},  \eqref{c3.11}-\eqref{4.32} and passing to the limits as $\varepsilon\rightarrow 0$, we finally obtain
\begin{equation}\label{c11}
	\begin{aligned}
		-\int_0^T\int_{\mathbb{T}^{n}} \phi_t \frac{1}{2}\rho |v|^2\,dxdt=0.	
	\end{aligned}
\end{equation}

Similar with the proof in Theorem \eqref{the1.2}, it suffices to establish the energy equality up to the initial time $t=0$. First, as the same manner of derivation of \eqref{c28} in Section \ref{sec4.1}, we have
\begin{equation}\label{4.35}
	\rho v\in C([0,T];L^{2}_{\text{weak}}(\mathbb{T}^n)).
\end{equation}
Moreover, from the mass equation $\eqref{NEuler}_1$, one has
\begin{equation}\label{4.36}
	\begin{aligned}
		\partial_t(\sqrt{\rho})=-v\cdot \nabla \sqrt{\rho},
	\end{aligned}
\end{equation}
which together with \eqref{key02} implies
$$\partial_t\sqrt{\rho}\in L^{3}(0,T;L^{\f{3n}{n+2}}(\mathbb{T}^n)),$$
and $$ \nabla \sqrt{\rho }\in L^{\infty}(0,T; L^{n}(\mathbb{T}^n)).$$
Hence, making use of the Aubin-Lions Lemma \ref{AL}, it holds
\begin{equation}\label{4.37}
\sqrt{\rho }\in C([0,T];L^{m}(\mathbb{T}^n)),\ \text{ for any} \ 1\leq m< \infty.
\end{equation}
Meanwhile, thanks to the energy inequality \eqref{inequa}, \eqref{4.35} and \eqref{4.37}, we have
\begin{equation}\label{c16}
	\begin{aligned}
		0&\leq \overline{\lim_{t\rightarrow 0}}\int_{\mathbb{T}^{n}} |\sqrt{\rho} v-\sqrt{\rho_0}v_0|^2 \,dx\\
		&=2\overline{\lim_{t\rightarrow 0}} \left(\int_{\mathbb{T}^{n}}\f{1}{2}\rho |v|^2 dx-\int_{\mathbb{T}^{n}}\f{1}{2}\rho_0 |v_0|^2\,dx\right)+2\overline{\lim_{t\rightarrow 0}}\int_{\mathbb{T}^{n}}\sqrt{\rho_0}v_0\left(\sqrt{\rho_0}v_0-\sqrt{\rho} v\right)\,dx\\
		&\leq 2\overline{\lim_{t\rightarrow 0}}\int_{\mathbb{T}^{n}} \sqrt{\rho_0}v_0\left(\sqrt{\rho_0}v_0-\sqrt{\rho}v\right)\,dx\\
		&=2\overline{\lim_{t\rightarrow 0}}\int_{\mathbb{T}^{n}} v_0 \left(\rho_0 v_0 -\rho v\right)\,dx+2\overline{\lim_{t\rightarrow 0}}\int_{\mathbb{T}^{n}} v_0 \sqrt{\rho }v\left(\sqrt{\rho }-\sqrt{\rho_0}\right)\,dx=0,
	\end{aligned}
\end{equation}
from
which it follows that
\begin{equation}\label{c17}
	\sqrt{\rho} v(t)\rightarrow \sqrt{\rho }v(0)\ \ {\rm strongly\ in}\ L^2(\mathbb{T}^n)\ {\rm as}\ t\rightarrow 0^+.
\end{equation}
Analogously, one has the right temporal continuity  of $\sqrt{\rho}v$ in $L^2(\mathbb{T}^n)$, namely, for any $t_0\geq 0$,
\begin{equation}\label{c18}
	\sqrt{\rho} v(t)\rightarrow \sqrt{\rho }v(t_0)\ \ {\rm strongly\ in}\ L^2(\mathbb{T}^n)\ {\rm as}\ t\rightarrow t_0^+.
\end{equation}
Finally, by repeating the process of \eqref{4.16}-\eqref{4.19}, we have
\begin{equation}\ba
	\int_{\mathbb{T}^{n}}\f{1}{2}\rho |v|^2(t_0)\,dx=&\int_{\mathbb{T}^{n}}\f{1}{2}\rho_0 |v_0|^2\,dx,
	\ea\end{equation}
which completes the proof of Theorem \ref{the1.3}.
\end{proof}

\section*{Acknowledgement}
Liu was partially supported by National Natural Science Foundation of China (No.  11801018), Beijing Natural Science Foundation (No. 1192001) and Beijing University of Technology (No. 006000514122514). Wang was partially supported by National Natural
Science Foundation of China (No. 11971446, No. 12071113 and No. 11601492). Ye was partially supported by National Natural
Science Foundation of China (No. 11701145).

\end{document}